\documentclass[article]{elsarticle}

\usepackage{lineno,hyperref,amsmath,amsthm,amssymb}
\modulolinenumbers[5]

\journal{Journal of Mathematical Analysis and Applications}

\newtheorem{thm}{Theorem}[section]
\newtheorem{lem}[thm]{Lemma}
\newdefinition{rmk}{Remark}
\newproof{pf}{Proof}
\newproof{pot}{Proof of Theorem \ref{thm2}}
\newtheorem{col}[thm]{Corollary}

\numberwithin{equation}{section}
\begin{document}

\begin{frontmatter}

\title{On the solution of a second order functional differential equation with a state derivative dependent delay}


\author[mymainaddress]{Jiraphorn Somsuwan}
\ead{jira.somsu@hotmail.com}

\author[mysecondaryaddress]{Keaitsuda Maneeruk Nakprasit\corref{mycorrespondingauthor}}
\cortext[mycorrespondingauthor]{Corresponding author Keaitsuda Maneeruk Nakprasit }
\ead{kmaneeruk@hotmail.com}

\address[mymainaddress]{Department of Mathematics, Faculty of Sciences, Khon Kaen University, Thailand}
\address[mysecondaryaddress]{Department of Mathematics, Faculty of Sciences, Khon Kaen University, Thailand}

\begin{abstract}
In this paper, the second order differential equation with a state derivative dependent delay of the form
$a_2x''(z) + a_1x'(z) + a_0x(z)  = x(p(z) + bx'(z))  + h(z)$ has been studied.
Considering a convergent power series $g(z)$ of an auxiliary equation
$a_2 \gamma^{2}  g''(\gamma z) g'(z)  = [g (\gamma^2 z) -  p(g(\gamma z))]  \gamma g'(\gamma z)(g' (z))^{2}
 + bh'(g(z))(g' (z))^{3} + \Big( a_2p''(g(z))+  a_1p'(g(z)) +a_0p(g(z))\Big) (g'(z))^{3}  - a_1\gamma g'(\gamma z) (g' (z))^{2} - a_0g(\gamma z)(g'(z))^{3} + a_2\gamma g'(\gamma z)g''( z)$  with the relation $p(z) + bx'(z) = g(\gamma g^{-1}(z)),$ we obtain an analytic solution $x(z).$
 Moreover, an analytic solution depends on a parameter $\gamma$ which satisfies one of the following conditions: $(H1) \ 0<|\gamma|<1,$ $(H2) \ \gamma = e^{2\pi i \theta }$
 where $\theta$ is a Brjuno number or $(H3) \ \gamma = e^{2\pi i \theta }$ where $\theta$ is a rational number.
\end{abstract}

\begin{keyword}
\texttt{functional differential equation\sep   analytic solution\sep Brjuno condition}
\MSC[2010] 34K05\sep  34A25 \sep 39B32
\end{keyword}

\end{frontmatter}

\linenumbers
\section{Introduction}
The functional differential equation
 $$x^{(m)}(z) = f(z, x^{(m_1)}(z - \tau _{1}(z)), x^{(m_2)}(z - \tau _{2}(z)), \ldots, x^{(m_k)}(z - \tau _{k}(z)))$$
where all  $m_i \geq 0, \tau_i \geq 0$, provides a mathematical model for a physical or biological system in which the rate of change of system is  determined not only by its present state, but also by its history (see \cite{BellmanCooke, Hale}.
In recent years, many authors studied the existence and the uniqueness of an analytic solution of a variety of these equations.
J. G. Si and S. S. Cheng \cite{SiCheng} and T. B. Liu and  H. Li \cite{Tongbo}
presented the existence of analytic solutions of functional differential equations
with a state derivative dependent delay $x'(z) = x(az + bx(z))$  and $  x''(z) +cx'(z)  =  x(az  + bx(z))$, respectively.

In this paper, we study the existence of analytic solutions of the second order differential equation with a state derivative dependent delay of the form
\begin{equation}\label{E:1}
a_2x''(z) + a_1x'(z) + a_0x(z)  = x(p(z) + bx'(z))  + h(z)
\end{equation}
where $a_0, a_1, a_2$ are  complex constants and $p(z), h(z)$ are analytic functions. In order to find analytic solutions of the  equation $(\ref{E:1})$, we construct a corresponding auxiliary equation with parameter $\gamma.$ The existence of solutions of an auxiliary equation depends on the condition of a parameter $\gamma$ that satisfies one of the following conditions: $\gamma$ is in the unit circle, $\gamma$ is a root of unity under Brjuno condition or $\gamma$ is not a root of unity.

First, we state the useful lemma which will be used in the case that parameter $\gamma$ satisfies a Brjuno condition.
Unless otherwise noted, we follow \cite{Brju} for terminology and notation. Let $\theta \in \mathbb{R}\setminus\mathbb{Q}$ and $\{p_n/q_n\}$ denote the sequence of continued fractions associated to $\theta.$

\begin{lem} \label{lemDV1}
$($Davie's lemma \cite{Brju}$).$ Let $$K(n) = n\log2 + \sum_{k=0}^{k(n)} g_k(n) \log(2q_{k+1}).$$
where $k(n)$ is defined by the condition $q_{k(n)} \leq n< q_{k(n)+1}$ and $g_k(n)$ is defined in \cite{Brju}. Then	
	\begin{itemize}
				\item[(a)] 		there is a positive constant $\xi > 0$ $($independent of $n$ and $\theta)$ such that
									   $$ K(n) \leq n ( B(\theta) + \xi )$$ where $B(\theta) = \sum_{n=0}^{\infty} \frac{\log q_{n+1}}{{q_n}},$
				\item[(b)] 		$ K(n_1) + K(n_2) \leq K(n_1 + n_2) $ for all $n_1$ and $n_2$,
				\item[(c)] 		$ -\log | \gamma^n -1 | \leq  K(n)  -  K(n-1).$
		\end{itemize}
\end{lem}

\section{Theorems}
 \indent In order to obtain an analytic solution of the equation $(\ref{E:1})$, we set
\begin{equation}\label{E:1.1}
	y(z) = p(z) + bx'(z).
\end{equation}
Then
$$ x(z) = x(z_{0}) + \frac{1}{b} \int_{z_{0}}^{z} (y(s) - p(s)) ds $$
where $z_0$ is a complex constant.
In particular, we  have
\begin{equation} \label{E:4.2}
	x(y(z)) = x(z_{0}) + \frac 1{b} \int_{z_{0}}^{y(z)} (y(s) - p(s)) ds.
\end{equation}
Applying relations $(\ref{E:1.1})$ and $(\ref{E:4.2})$ to $(\ref{E:1})$, we obtain
\begin{eqnarray} \label{E:5}
\frac{a_2}{b} \Big( y'(z) - p'(z) \Big) + \frac{a_1}{b}\Big( y(z) - p(z)\Big) +a_0(x(z_{0}) + \frac{1}{b} \int_{z_{0}}^{z} (y(s) - p(s)) ds )  \nonumber \\
= x(z_{0}) + \frac{1}{b} \int_{z_{0}}^{y(z)} (y(s) - p(s)) ds + h(z).
\end{eqnarray}
We construct the corresponding equation by differentiating both sides of $(\ref{E:5})$ with respect to $z.$
This yields
\begin{align} \label{E:6}
  a_2(y''(z) - p''(z)) +a_1(y'(z) - p'(z)) + a_0(y(z) - p(z)) \nonumber \\
= (y(y(z)) - p(y(z)))y'(z) + bh'(z).
\end{align}

Consider the auxiliary equation
 \begin{align} \label{E:2}
a_2 \gamma^{2}  g''(\gamma z) g'(z)  &= [g (\gamma^2 z) -  p(g(\gamma z))]  \gamma g'(\gamma z)(g' (z))^{2}
 + bh'(g(z))(g' (z))^{3} \nonumber \\
& \qquad  + \Big( a_2p''(g(z))+  a_1p'(g(z)) +a_0p(g(z))\Big) (g'(z))^{3} \nonumber \\
& \qquad - a_1\gamma g'(\gamma z) (g' (z))^{2} - a_0g(\gamma z)(g'(z))^{3} + a_2\gamma g'(\gamma z)g''( z)
\end{align}
 where $g(z)$ satisfies the initial value conditions $g(0) = 0$ and $g'(0) = \eta \not= 0$ with $\gamma, \eta$ are complex numbers.
 Since $g'(0) =\eta \not= 0$, the equation $(\ref{E:2})$ can be reduced equivalently to the equation
  \begin{align}\label{E:3}
  a_2\gamma g'(\gamma z)  &=  g'(z)\Bigg[a_2\gamma +  \int_{0}^{z} \Big[ \Big (g (\gamma^2 s) -  p(g(\gamma s)) \Big)  \gamma g'(\gamma s)
+ a_2p''(g(s))g'(s) \nonumber  \\
& \qquad \qquad +  a_1p'(g(s))g'(s) + a_0p(g(s))g'(s) - a_1\gamma g'(\gamma s) - a_0 g(\gamma s)g'(s) \nonumber \\
& \qquad \qquad + bh'(g(s))g' (s) \Big]  ds\Bigg]
  \end{align}
 with $g(0)=0$ and $g'(0) = \eta \not=0.$
To construct analytic solutions of $(\ref{E:3})$, we separate our study on the conditions of the parameter $\gamma$ as follows:
\begin{itemize}
	\item[$(H1)$]		$ 0 < |\gamma| < 1;$
	\item[$(H2)$]		$\gamma$ is not a root of unity in the form $\gamma = e^{2 \pi i \theta}$ where $\theta \in \mathbb{R} \setminus \mathbb{Q}$ is a Brjuno number. That is, $B(\theta) = \sum_{n=0}^{\infty} \frac{\log q_{n+1}}{{q_n}} < \infty,$ where $\{p_n/q_n \}$ denotes the sequence of of continued fractions associated to $\theta.$
	\item[$(H3)$]		$\gamma$ is a root of unity in the form $\gamma = e^{2 \pi i q/p}$ where $p \in  \mathbb{N}$ with $p \geq 2$ and $q \in \mathbb{Z} \setminus \{ 0 \},$ and $\gamma \not= e^{2 \pi il/k}$ for all $1 \leq  k \leq p-1$ and $l \in \mathbb{Z} \setminus \{0\}.$
\end{itemize}
From now on, we let $p(z)$ and $h(z)$ be  analytic functions in a neighborhood of the origin.
Then we represent $p(z)$ and $h(z)$ by  power series $\sum_{n=0}^{\infty} p_nz^n$ and  $\sum_{n=0}^{\infty} h_nz^n,$ respectively.

\begin{thm}\label{T:1}
Let $\gamma$ satisfy the condition $(H1)$. Then the equation $(\ref{E:2})$ has an analytic solution
 \begin{equation}  \label{E:4}
g(z) = \sum_{n=1}^{\infty} c_nz^n
\end{equation}
  in a neighborhood of the origin such that $g(0) = 0, g'(0) = \eta$ where $\eta$ is a nonzero complex number.
\end{thm}

\begin{proof}
Since $p(z)$ is analytic in a neighborhood of the origin, there exists a constant $\rho$ such that $| p_n | \leq  \rho ^{n-1}$ for $n \geq 1$.
 Substituting $(\ref{E:4})$  into  $(\ref{E:3})$ and comparing coefficients of $z^n (n=1,2,\ldots)$, we get
$$  a_2 c_1 \gamma =  a_2  c_1\gamma,  2 a_2 c_2\gamma^2    = 2 a_2 c_2 \gamma  - p_0c_1^2 \gamma + 2a_2p_2c_1^2  + a_1 p_1 c_1^2 + p_0a_0 c_1^2 - a_1 c_1^2 \gamma  + bh_1 c_1^2, $$
 and  in general for $n \geq 1$
 \begin{eqnarray}
 &{}&  \!\!\!\!\!\!\!\!\! \!\!\!\!\!\!\!\!\!
 (n+2)a_2(\gamma^{n+2} - \gamma )c_{n+2}  \nonumber \\
  &=&  \sum_{k=1}^{n}   \sum_{i=1}^{k} \frac{i (n-k+1)\gamma^{2k-i+2}}{k+1} \cdot  c_i c_{k-i+1}  c_{n-k+1}
  -    \sum_{k=1}^{n+1} k p_0 c_k c_{n-k+2} \gamma^{n-k+2}    \nonumber  \\
&{}&  - \sum_{k=1}^{n}   \sum_{i=1}^{k} \frac{i (n-k+1)\gamma^{k+1}}{k+1} \cdot c_i c_{n-k+1}   \quad  \sum _{\makebox [0 pt ]{$%
{{\scriptscriptstyle l_1+\cdots+l_m = k-i+1 \atop {%
{m=1, \ldots,k-i+1 }}}} %
$}} p_m c_{l_1}\cdots c_{l_m}  +  \sum_{k=1}^{n+1} 2k p_2 a_2 c_k c_{n-k+2}   \nonumber \\
&{}& +   \sum_{k=1}^{n}   \sum_{i=1}^{k} \frac{i (n-k+1)a_2}{k+1} \cdot c_i c_{n-k+1} \quad \sum _{\makebox [0 pt ]{$%
{{\scriptscriptstyle l_1+\cdots+l_m = k-i+1 \atop {%
{m=1, \ldots,k-i+1 }}}} %
$}}  (m+2)(m+1)p_{m+2} c_{l_1}\cdots c_{l_m}   \nonumber \\
&{}& + \sum_{k=1}^{n+1} k  p_1 a_1  c_kc_{n-k+2}
 +  \sum_{k=1}^{n}   \sum_{i=1}^{k} \frac{ i (n-k+1)a_1}{k+1} \cdot c_i c_{n-k+1} \quad  \sum _{\makebox [0 pt ]{$%
{{\scriptscriptstyle l_1+\cdots+l_m = k-i+1 \atop {%
{m=1, \ldots,k-i+1 }}}} %
$}} (m+1)p_{m+1} c_{l_1}\cdots c_{l_m}   \nonumber \\
&{}&  +  \sum_{k=1}^{n+1} k p_0 a_0 c_k c_{n-k+2}
 +  \sum_{k=1}^{n}   \sum_{i=1}^{k} \frac{i (n-k+1)a_0 }{k+1} \cdot c_i c_{n-k+1} \quad  \sum _{\makebox [0 pt ]{$%
{{\scriptscriptstyle l_1+\cdots+l_m = k-i+1 \atop {%
{m=1,  \ldots,k-i+1 }}}} %
$}}  p_{m} c_{l_1}\cdots c_{l_m} \nonumber \\
&{}&  -   \sum_{k=1}^{n+1}  k a_1  c_{k}  c_{n-k+2} \gamma^{n-k+2}
 -   \sum_{k=1}^{n}  \sum_{i=1}^{k} \frac{ i(n-k+1)a_0}{k+1} \cdot c_i c_{k-i+1}c_{n-k+1} \gamma^{k-i+1}   \nonumber \\
&{}&  + \sum_{k=1}^{n+1} bkh_1 c_k  c_{n-k+2}
 + \sum_{k=1}^{n} \sum_{i=1}^{k}  \frac{bi(n-k+1)}{k+1}\cdot c_i c_{n-k+1} \quad \sum _{\makebox [0 pt ]{$%
{{\scriptscriptstyle l_1+\cdots+l_m = k-i+1 \atop {%
{m=1, \ldots,k-i+1 }}}} %
$}}  (m+1)h_{m+1} c_{l_1} \cdots c_{l_m}.\nonumber
	 \end{eqnarray}
The first expression allows us to choose $c_1 = \eta \not= 0$  and the second expression implies
 $c_2 =\frac{p_0a_0 + p_1a_1   + 2p_2a_2 + bh_1 - p_0 \gamma  - a_1 \gamma }{2a_2 \gamma (\gamma - 1)} \cdot c_1^{2}. $
 Consequently the sequence  $\{ c_n\}_{n=3}^{\infty}$ is successively determined by the last expression in a unique manner. This implies that $(\ref{E:2})$ has a formal power  series solution. Next, we show that the power series $g(z)$ converges in a neighborhood of the origin.
Since $| p_n | \leq  \rho^{n-1}$ and $\lim_{n \to \infty} \frac{1}{\gamma^{n+1} -1} = -1$ for  $0 < |\gamma| < 1$, there exists a positive constant  $M$ such that
\begin{align}
|c_{n+2}|   &\leq   M \Bigg[  2\sum_{k=1}^{n}   \sum_{i=1}^{k} |c_i|  |c_{k-i+1}|  |c_{n-k+1}| +  6\sum_{k=1}^{n+1} |c_k| |c_{n-k+2}|    \nonumber \\
 & \qquad  \quad +  2\sum_{k=1}^{n}   \sum_{i=1}^{k} | c_i|   |c_{n-k+1}|
  \quad  \sum _{\makebox [0 pt ]{$%
{{\scriptscriptstyle l_1+\cdots+l_m = k-i+1 \atop {%
{m=1,  \ldots,k-i+1 }}}} %
$}}  |c_{l_1}| \cdots |c_{l_m}|   \nonumber \\
&  \qquad  \quad + 2\sum_{k=1}^{n}   \sum_{i=1}^{k} | c_i|  |c_{n-k+1}| \quad \sum _{\makebox [0 pt ]{$%
{{\scriptscriptstyle l_1+\cdots+l_m = k-i+1 \atop {%
{m=1, \ldots,k-i+1 }}}} %
$}} (m+1) |c_{l_1}|\cdots |c_{l_m}|  \nonumber \\
& \qquad  \quad +  \sum_{k=1}^{n}   \sum_{i=1}^{k} | c_i|  |c_{n-k+1}|
 \quad \sum _{\makebox [0 pt ]{$%
{{\scriptscriptstyle l_1+\cdots+l_m = k-i+1 \atop {%
{m=1, \ldots,k-i+1 }}}} %
$}}  (m+2)(m+1) |c_{l_1}|\cdots |c_{l_m}|   \Bigg].  \nonumber
	 \end{align}
Let us define a power series $\sum_{n=1}^{\infty} d_nz^n$  where a positive sequence $\{d_n\}_{n=1}^{\infty}$ is determined  by $d_1 = |\eta|,
d_2 =\Big|\frac{p_0a_0 + p_1a_1   + 2p_2a_2 + bh_1 - p_0 \gamma  - a_1 \gamma  }{2a_2 \gamma (\gamma - 1)} \cdot c_1^{2}\Big|$
and for $n \geq  1$
\begin{align*}
d_{n+2}   &=   M \Bigg[  2\sum_{k=1}^{n}   \sum_{i=1}^{k} d_i  d_{k-i+1}  d_{n-k+1}   +  6\sum_{k=1}^{n+1} d_k d_{n-k+2} \nonumber \\
 & \qquad \qquad +   2\sum_{k=1}^{n}   \sum_{i=1}^{k} d_i   d_{n-k+1}
  \quad  \sum _{\makebox [0 pt ]{$%
{{\scriptscriptstyle l_1+\cdots+l_m = k-i+1 \atop {%
{m=1, \ldots,k-i+1 }}}} %
$}}  d_{l_1} \cdots d_{l_m}  \nonumber \\
& \qquad \qquad +   2\sum_{k=1}^{n}   \sum_{i=1}^{k} d_i  d_{n-k+1}
 \quad \sum _{\makebox [0 pt ]{$%
{{\scriptscriptstyle l_1+\cdots+l_m = k-i+1 \atop {%
{m=1, \ldots,k-i+1 }}}} %
$}} (m+1) d_{l_1}\cdots d_{l_m}  \nonumber \\
& \qquad \qquad +   \sum_{k=1}^{n}   \sum_{i=1}^{k} d_i  d_{n-k+1}
 \quad \sum _{\makebox [0 pt ]{$%
{{\scriptscriptstyle l_1+\cdots+l_m = k-i+1 \atop {%
{m=1, \ldots,k-i+1 }}}} %
$}}  (m+2)(m+1) d_{l_1}\cdots d_{l_m}\Bigg].
	 \end{align*}
It follows that $|c_n| \leq d_n$ for $n \geq 1$. That is, $\sum_{n=1}^{\infty} d_nz^n$  is a majorant series of $\sum_{n=1}^{\infty} c_nz^n.$
We show that $\sum_{n=1}^{\infty} d_nz^n$ is analytic in a neighborhood of the origin.
Note that if we let $D(z) = \sum_{n=1}^{\infty} d_nz^n,$ then
\begin{align*}
D(z) &= \sum_{n=1}^{\infty} d_nz^n  \\
&=   |\eta| z + d_2 z^2 +  \sum_{n=1}^{\infty} d_{n+2}z^{n+2}\\
&=   |\eta| z + d_2 z^2 + M \left[  2\sum_{n=1}^{\infty}\left( \sum_{k=1}^{n}   \sum_{i=1}^{k} d_i  d_{k-i+1}  d_{n-k+1} \right) z^{n+2}
 +  6\sum_{n=1}^{\infty}\left( \sum_{k=1}^{n+1} d_k d_{n-k+2} \right)z^{n+2} \right.  \\
 & \left. \quad +   2\sum_{n=1}^{\infty} \left( \sum_{k=1}^{n}   \sum_{i=1}^{k} d_i   d_{n-k+1}
  \quad  \sum _{\makebox [0 pt ]{$%
{{\scriptscriptstyle l_1+\cdots+l_m = k-i+1 \atop {%
{m=1,  \ldots,k-i+1 }}}} %
$}}  d_{l_1} \cdots d_{l_m}   \right) z^{n+2}   \right.  \\
& \quad \left. +   2\sum_{n=1}^{\infty} \left( \sum_{k=1}^{n}   \sum_{i=1}^{k} d_i  d_{n-k+1}
 \quad   \sum _{\makebox [0 pt ]{$%
{{\scriptscriptstyle l_1+\cdots+l_m = k-i+1 \atop {%
{m=1, \ldots,k-i+1 }}}} %
$}} (m+1) d_{l_1}\cdots d_{l_m}  \right) z^{n+2}  \right.  \\
& \left. \quad +   \sum_{n=1}^{\infty} \left( \sum_{k=1}^{n}   \sum_{i=1}^{k} d_i  d_{n-k+1}
 \quad   \sum _{\makebox [0 pt ]{$%
{{\scriptscriptstyle l_1+\cdots+l_m = k-i+1 \atop {%
{m=1, \ldots,k-i+1 }}}} %
$}}  (m+2)(m+1) d_{l_1}\cdots d_{l_m}  \right) z^{n+2} \right] \\
&=   |\eta| z + d_2  z^2 + M \left[ 2(D(z))^3 + 6(D(z))^2  - 6|\eta|^2z^2  + \frac{2(D(z))^3}{1-D(z)}  \right. \\
& \left. \qquad +  \frac{2(2-D(z)) (D(z))^3}{(D(z) -1 )^2} + \frac{2(D(z))^3 (3 - 3D(z) + (D(z))^2)}{(1 - D(z))^3}  \right].
	 \end{align*}
Consider  the  equation
\begin{align*}
T(z,D(z))
&= D(z) -  |\eta| z - d_2 z^2 -  M \left[ 2(D(z))^3 + 6(D(z))^2  - 6|\eta|^2z^2  + \frac{2(D(z))^3}{1-D(z)}  \right.  \\
& \left. \qquad +  \frac{2(2-D(z)) (D(z))^3}{(D(z) -1 )^2} + \frac{2(D(z))^3 (3 - 3D(z) + (D(z))^2)}{(1 - D(z))^3}  \right]  = 0.
 \end{align*}
 Since $T$ is continuous in a neighborhood of the origin, $T(0,0) = 0$ and $ T'_D(0,0) = 1 \not=0$, the implicit function theorem implies that
there exists a unique function $ D(z)$ which is analytic in a neighborhood of the origin
with a positive radius. Because $D(z)$ is a majorant series of $g(z),$   $g(z)$ is also analytic in a neighborhood of the origin with a positive radius.
This completes the proof.
\end{proof}

Now, we consider an analytic solution $g(z)$ of the auxiliary equation $(\ref{E:2})$ in the case of $\gamma$ satisfies the condition $(H2).$

\begin{thm} \label{T:2}
Assume  that  $\gamma$  satisfies  the condition   $(H2)$. Then  there  exists  an  analytic  solution
$$g(z) = \sum_{n=1}^{\infty} c_n z^n$$  of the  equation $(\ref{E:2})$  in  a  neighborhood  of  the  origin  such  that  $g(0) = 0, g'(0) =  \eta $  where  $\eta$  is  a  nonzero  complex  number.
\end{thm}

\begin{proof}
We now imitate the proof of Theorem $\ref{T:1}$ with approximate new bound. The sequence $\{c_n \}_{n=1}^{\infty}$ is defined similar to the proof of Theorem $\ref{T:1}.$ Note that  $c_1 = \eta \not= 0$ and $c_2 =\frac{p_0a_0 + p_1a_1   + 2p_2a_2 + bh_1 - p_0 \gamma  - a_1 \gamma  }{2a_2 \gamma (\gamma - 1)} \cdot c_1^{2}.$
Since $|\gamma| = 1$ and $p(z)$ is analytic near the origin, there exists a positive constant $N$ so that for $n \geq 1$
\begin{align*}
|c_{n+2}|   &\leq  \frac{N}{|\gamma^{n+1} - 1 |}  \Bigg[  2\sum_{k=1}^{n}   \sum_{i=1}^{k} |c_i|  |c_{k-i+1}|  |c_{n-k+1}|
+  6\sum_{k=1}^{n+1} |c_k| |c_{n-k+2}|  \nonumber \\
 & \qquad  +   2\sum_{k=1}^{n}   \sum_{i=1}^{k} | c_i|   |c_{n-k+1}|
  \quad  \sum _{\makebox [0 pt ]{$%
{{\scriptscriptstyle l_1+\cdots+l_m = k-i+1 \atop {%
{m=1, \ldots,k-i+1 }}}} %
$}}
  |c_{l_1}| \cdots |c_{l_m}|  \nonumber \\
&  \qquad +   2\sum_{k=1}^{n}   \sum_{i=1}^{k} | c_i|  |c_{n-k+1}|
 \quad \sum _{\makebox [0 pt ]{$%
{{\scriptscriptstyle l_1+\cdots+l_m = k-i+1 \atop {%
{m=1,  \ldots,k-i+1 }}}} %
$}}
  (m+1) |c_{l_1}|\cdots |c_{l_m}|\nonumber \\
& \qquad +   \sum_{k=1}^{n}   \sum_{i=1}^{k} | c_i|  |c_{n-k+1}|
 \quad \sum _{\makebox [0 pt ]{$%
{{\scriptscriptstyle l_1+\cdots+l_m = k-i+1 \atop {%
{m=1,  \ldots,k-i+1 }}}} %
$}}
  (m+2)(m+1) |c_{l_1}|\cdots |c_{l_m}|  \Bigg].
	 \end{align*}
To construct a governing series of $g(z)$, we let $\{ D_n\}_{n=1}^{\infty}$ be a nonnegative sequence determined by $ D_1 = |\eta|,
D_2=\Big|\frac{p_0a_0 + p_1a_1   + 2p_2a_2 + bh_1 - p_0 \gamma  - a_1 \gamma }{2a_2  (\gamma - 1)} \cdot c_1^{2}\Big|$ and for all $n \geq 1$
	 \begin{align*}
D_{n+2} &= N  \Bigg[  2\sum_{k=1}^{n}   \sum_{i=1}^{k} D_i D_{k-i+1} D_{n-k+1}
+  6 \sum_{k=1}^{n+1} D_k D_{n-k+2}  \nonumber \\
 & \qquad \quad  +   2\sum_{k=1}^{n}   \sum_{i=1}^{k} D_i D_{n-k+1}
  \quad  \sum _{\makebox [0 pt ]{$%
{{\scriptscriptstyle l_1+\cdots+l_m = k-i+1 \atop {%
{m=1,  \ldots,k-i+1 }}}} %
$}}
  D_{l_1} \cdots D_{l_m}  \nonumber \\
& \qquad \quad  +   2\sum_{k=1}^{n}   \sum_{i=1}^{k} D_i D_{n-k+1}
 \quad \sum _{\makebox [0 pt ]{$%
{{\scriptscriptstyle l_1+\cdots+l_m = k-i+1 \atop {%
{m=1,  \ldots,k-i+1 }}}} %
$}} (m+1) D_{l_1}\cdots D_{l_m}    \nonumber \\
&  \qquad \quad  +   \sum_{k=1}^{n}   \sum_{i=1}^{k} D_i D_{n-k+1}
 \quad \sum _{\makebox [0 pt ]{$%
{{\scriptscriptstyle l_1+\cdots+l_m = k-i+1 \atop {%
{m=1,  \ldots,k-i+1 }}}} %
$}}
  (m+2)(m+1) D_{l_1}\cdots D_{l_m} \Bigg].
	 \end{align*}
From this construction, we can demonstate that a power series $ D(z) = \sum_{n=1}^{\infty} D_nz^n$ satisfies the implicit functional equation
\begin{align*}
R(z, D(z)) &= D(z)  -  |\eta| z - D_2 z^2 -  N \left[ 2(D(z))^3 + 6(D(z))^2  - 6|\eta|^2z^2  + \frac{2(D(z))^3}{1-D(z)}  \right. \nonumber \\
& \left. \qquad + \frac{ 2(2-D(z)) (D(z))^3}{(D(z) -1 )^2} + \frac{2(D(z))^3 (3 - 3D(z) + (D(z))^2)}{(1 - D(z))^3}  \right]  \nonumber \\
&= 0
\end{align*}
with $R(0,0) =0$ and $R'_D(0,0) = 1 \not= 0.$ This yields the power series $D(z)$ converges in a neighborhood of the origin. Hence, there exists a positive constant $T$ such that $ D_n \leq T^n$ for $n \geq 1.$\\ \indent
Let  $K$ be a function defined as in Lemma $\ref{lemDV1}.$ By mathematical induction, we can show that for $n \in \mathbb{N} \cup \{ 0 \} $
$$ | c_{n+1} | \leq D_{n+1}e^{K(n)}. $$
Lemma $\ref{lemDV1}$ yields $\lim_{n \to \infty} (|c_{n+1}|)^{\frac{1}{n}}  \leq  T e^{B(\theta) + \xi }.$
This implies that $g(z)$ has a convergence radius at least $(T e^{B(\theta) + \xi })^{-1}.$ The proof is completed.
\end{proof}

Finally, we consider the case of $\gamma$ satisfies the condition $(H3).$ In this case, $\gamma$ is not only on the unit circle, but also a root of unity.
Let $\{ R_n \}_{n=1}^{\infty}$ be a sequence defined by $R_1 = | \eta|, R_2 = \Gamma A$ with
$ A =\frac{|p_0a_0 + p_1 a_1  + 2p_2a_2 + bh_1 - p_0 \gamma  - a_1 \gamma  | |\eta^2|}{|2a_2  \gamma|}$,
  $\Gamma = \max \Big\{ \frac{1}{| \gamma -1 |}, \frac{1}{| \gamma^2 -1 |}, \ldots, \frac{1}{| \gamma^{p-1} -1 |}  \Big\},$	
and
\begin{align}
R_{n+2}  &= \Gamma N  \Bigg[ 2 \sum_{k=1}^{n}   \sum_{i=1}^{k} R_i  R_{k-i+1}  R_{n-k+1}   +  6 \sum_{k=1}^{n+1} R_k R_{n-k+2} \nonumber \\
 &  \qquad  \qquad +2\sum_{k=1}^{n}   \sum_{i=1}^{k} R_i   R_{n-k+1}
  \quad  \sum _{\makebox [0 pt ]{$%
{{\scriptscriptstyle l_1+\cdots+l_m = k-i+1 \atop {%
{m=1,  \ldots,k-i+1 }}}} %
$}}
 R_{l_1}\cdots R_{l_m} \nonumber \\
& \qquad  \qquad +  2\sum_{k=1}^{n}   \sum_{i=1}^{k} R_i  R_{n-k+1}
 \quad \sum _{\makebox [0 pt ]{$%
{{\scriptscriptstyle l_1+\cdots+l_m = k-i+1 \atop {%
{m=1,  \ldots,k-i+1 }}}} %
$}} (m+1) R_{l_1}\cdots R_{l_m} \nonumber \\
&  \qquad  \qquad +   \sum_{k=1}^{n}   \sum_{i=1}^{k} R_i  R_{n-k+1}
 \quad \sum _{\makebox [0 pt ]{$%
{{\scriptscriptstyle l_1+\cdots+l_m = k-i+1 \atop {%
{m=1,  \ldots,k-i+1 }}}} %
$}}
  (m+2)(m+1) R_{l_1}\cdots R_{l_m}  \Bigg]  \label{E:10}
\end{align}
where $N$ is a positive constant defined as in the proof of Theorem $\ref{T:2}.$
\begin{thm} \label{T:3}
Assume  that  $\gamma$  satisfies the condition $(H3).$ Let  $g(z)  =\sum_{n=1}^{\infty} c_nz^n$  be  a power  series determined  by
$c_1 = \eta \not=0,  c_2 =\frac{(p_0a_0 +  p_1 a_1 + 2p_2a_2 + bh_1 - p_0 \gamma  - a_1 \gamma  ) c_1^2}{2 a_2 \gamma(\gamma -1)} $
and
\begin{eqnarray*}
(n+2) a_2 \gamma(\gamma^{n+1} - 1 )c_{n+2} = \Theta (n, \gamma), n = 1, 2, \ldots
\end{eqnarray*}
where
\begin{eqnarray*}
\Theta (n, \gamma)
  &=&  \sum_{k=1}^{n}   \sum_{i=1}^{k} \frac{i (n-k+1)\gamma^{2k-i+2}}{k+1} \cdot  c_i c_{k-i+1}  c_{n-k+1}
  - \sum_{k=1}^{n+1} kp_0 c_k c_{n-k+2} \gamma^{n-k+2}  \\
&{}&  - \sum_{k=1}^{n}   \sum_{i=1}^{k} \frac{i (n-k+1)\gamma^{k+1}}{k+1} \cdot c_i c_{n-k+1}   \quad  \sum _{\makebox [0 pt ]{$%
{{\scriptscriptstyle l_1+\cdots+l_m = k-i+1 \atop {%
{m=1,  \ldots,k-i+1 }}}} %
$}} p_m c_{l_1}\cdots c_{l_m}  \\
&{}& +  \sum_{k=1}^{n+1} 2kp_2a_2 c_k c_{n-k+2}   \\
&{}& +   \sum_{k=1}^{n}   \sum_{i=1}^{k} \frac{i (n-k+1)a_2}{k+1} \cdot c_i c_{n-k+1} \quad \sum _{\makebox [0 pt ]{$%
{{\scriptscriptstyle l_1+\cdots+l_m = k-i+1 \atop {%
{m=1,  \ldots,k-i+1 }}}} %
$}}  (m+2)(m+1)p_{m+2} c_{l_1}\cdots c_{l_m}
\end{eqnarray*}
\begin{eqnarray*}
&{}& + \sum_{k=1}^{n+1}  k p_1 a_1  c_kc_{n-k+2}
 +  \sum_{k=1}^{n}   \sum_{i=1}^{k} \frac{ i (n-k+1)a_1}{k+1} \cdot c_i c_{n-k+1} \quad  \sum _{\makebox [0 pt ]{$%
{{\scriptscriptstyle l_1+\cdots+l_m = k-i+1 \atop {%
{m=1,  \ldots,k-i+1 }}}} %
$}} (m+1)p_{m+1} c_{l_1}\cdots c_{l_m}   \\
&{}&  + \sum_{k=1}^{n+1} kp_0 a_0 c_k c_{n-k+2}
 +  \sum_{k=1}^{n}   \sum_{i=1}^{k} \frac{ i (n-k+1)a_0}{k+1} \cdot c_i c_{n-k+1} \quad  \sum _{\makebox [0 pt ]{$%
{{\scriptscriptstyle l_1+\cdots+l_m = k-i+1 \atop {%
{m=1,  \ldots,k-i+1 }}}} %
$}}  p_{m} c_{l_1}\cdots c_{l_m}\\
&{}&  -   \sum_{k=1}^{n+1}   k a_1  c_{k}  c_{n-k+2} \gamma^{n-k+2}
 -   \sum_{k=1}^{n}  \sum_{i=1}^{k} \frac{ i(n-k+1)a_0}{k+1} \cdot c_i c_{k-i+1}c_{n-k+1} \gamma^{k-i+1}  \\
&{}&  + \sum_{k=1}^{n+1}  bkh_1 c_{n-k+2} c_k
  + \sum_{k=1}^{n} \sum_{i=1}^{k}  \frac{bi(n-k+1)}{k+1} \cdot c_i c_{n-k+1} \quad \sum _{\makebox [0 pt ]{$%
{{\scriptscriptstyle l_1+\cdots+l_m = k-i+1 \atop {%
{m=1, \ldots,k-i+1 }}}} %
$}}  (m+1)h_{m+1} c_{l_1} \cdots c_{l_m}.
	 \end{eqnarray*}

If  $\Theta (vp -1, \gamma)= 0$  for  $ v=1,2, \ldots,$ then the equation $(\ref{E:2})$  has  an  analytic  solution  $g(z)$  in  a neighborhood  of  the  origin
such  that  $g(0) =0, g'(0) = \eta \not=0.$

Otherwise, if $\Theta (vp -1, \gamma) \not= 0$  for  some  $v=1,2,\ldots,$ then the equation  $(\ref{E:2})$  has  no  analytic  solution  in   a neighborhood  of  the origin.
\end{thm}
\begin{proof}
Observe  that  if   $\eta = 0$ then  $g(z) \equiv  0$ which is  a  trivial  analytic  solution.
So  we  consider  only  the  case  $\eta \not= 0.$\\
\indent
If $\Theta (vp -1, \gamma) \not= 0$  for  some  positive  number  $v$,  then $ (vp + 1) \gamma(\gamma^{vp} - 1 )c_{vp+1}   \not = 0.$
But  the condition $(H3)$ implies  $\gamma^{vp} - 1 = 0,$ which  is  a  contradiction.
This concludes  that the equation  $(\ref{E:2})$  has  no  analytic  solution in a neighborhood of the origin.\\
\indent
Assume that $\Theta (vp -1, \gamma)= 0$ for $ v=1,2, \ldots.$ Then $(vp+1) \gamma(\gamma^{vp} - 1 )c_{vp+1}  =  0.$
So, there are infinitely many choices of $c_{vp+1}.$  Choose $c_{vp+1} = \eta_{vp+1}$ such that $|\eta_{vp+1}|  \leq  R_{vp+1}$ where $R_{vp+1}$ is defined in $(\ref{E:10}).$\\
Note that $| \gamma^{n+1} |^{-1} \leq \Gamma $ for $n \not= vp -1$ where $\Gamma = \max \left\{\frac{1}{|\gamma -1|}, \frac{1}{|\gamma^2-1|} ,\ldots,\frac{1}{ |\gamma^{p-1} -1|}  \right \}.$
One can see that
\begin{align*}
|c_{n+2}|   &\leq  \Gamma  N \Bigg[  2\sum_{k=1}^{n}   \sum_{i=1}^{k} |c_i|  |c_{k-i+1}|  |c_{n-k+1}|
+  6\sum_{k=1}^{n+1} |c_k| |c_{n-k+2}| \\
 & \qquad \qquad + 2\sum_{k=1}^{n}   \sum_{i=1}^{k} | c_i|   |c_{n-k+1}|
  \quad \sum _{\makebox [0 pt ]{$%
{{\scriptscriptstyle l_1+\cdots+l_m = k-i+1 \atop {%
{m=1, \ldots,k-i+1 }}}} %
$}}
  |c_{l_1}| \cdots |c_{l_m}| \\
& \qquad \qquad + 2\sum_{k=1}^{n}   \sum_{i=1}^{k} | c_i|  |c_{n-k+1}|
 \quad \sum _{\makebox [0 pt ]{$%
{{\scriptscriptstyle l_1+\cdots+l_m = k-i+1 \atop {%
{m=1, \ldots,k-i+1 }}}} %
$}} (m+1) |c_{l_1}|\cdots |c_{l_m}| \\
& \qquad \qquad +   \sum_{k=1}^{n}   \sum_{i=1}^{k} | c_i|  |c_{n-k+1}|
 \quad \sum _{\makebox [0 pt ]{$%
{{\scriptscriptstyle l_1+\cdots+l_m = k-i+1 \atop {%
{m=1, \ldots,k-i+1 }}}} %
$}}
  (m+2)(m+1) |c_{l_1}|\cdots |c_{l_m}| \Bigg]
	 \end{align*}
for $n \not= vp - 1, v=1, 2, \ldots. $ \\ \indent
Likewise, the remaining proof is similar to one of Theorem $\ref{T:1}.$ Consider the implicit functional equation
\begin{align}  \label{E:12}
H(z,R(z))
&= R(z) -  |\eta| z - A z^2 - N \left[ 2(R(z))^3 + 6(R(z))^2  - 6|\eta|^2 z^2  + \frac{2(R(z))^3}{1-R(z)}  \right. \nonumber \\
& \left. \qquad +  \frac{2(2-R(z)) (R(z))^3}{(R(z) -1 )^2} + \frac{2(R(z))^3 (3 - 3R(z) + (R(z))^2)}{(1 - R(z))^3}  \right]  \nonumber \\
&  = 0.
 \end{align}
Since $H(0,0)=0, H'_R(0,0) = 1 \not=0$, the implicit function theorem implies that there exists a unique function $R(z)$ which is analytic in a neighborhood of the origin with a positive radius. We can show that the power series  $R(z) = \sum_{n=1}^{\infty} R_n z^n$ which $\{R_n\}_{n=1}^{\infty}$ is determined by $(\ref{E:10})$ satisfies the equation $(\ref{E:12}).$ Moreover, $|c_n| \leq R_n$ for $n \geq 1.$ That is, $R(z)$ is a majorant series of $g(z).$ Then $g(z)$ converges in a neighborhood of the origin. This completes the proof.
\end{proof}
\section{Results and Discussion}
\begin{thm} \label{T:4}
Let  $g(z) = \sum_{n=1}^{\infty} c_n z^n$ be an analytic solution in a neighborhood of the origin of the equation $(\ref{E:2}),$ with $g(0) =0, g'(0) = \eta \not= 0,$ which is obtained from Theorem $\ref{T:1}$, Theorem $\ref{T:2}$, or Theorem $\ref{T:3}.$
Then the equation $(\ref{E:6})$ has an analytic solution of the form $y(z)= g(\gamma g^{-1}(z))$ in a neighborhood of the origin.
\end{thm}
\begin{proof}
Since  $ g'(0) = \eta \not= 0,  g^{-1}(z)$  is analytic in a neighborhood of $g(0) = 0$.  \\
Let $y(z)  =  g(\gamma g^{-1}(z)).$ Then
\begin{eqnarray*}
&{}&  \!\!\!\!\!\!\!\!\! \!\!\!\!\!\!\!\!\! \!\!\!  a_2\Big( y''(z) - p''(z) \Big) +  a_1\Big( y'(z) - p'(z)\Big) + a_0\Big( y(z) -p(z)\Big) \nonumber  \\
  &=& \frac{1}{(g^{-1}(z)))^{3}} \Bigg(  a_2 \gamma^{2}  g''(\gamma g^{-1}(z)) g' (g^{-1}(z)) - a_2 \gamma g'(\gamma g^{-1}(z)) g''( g^{-1}(z))  \nonumber \\
 &{}& \qquad - a_2p''(z)(g' (g^{-1}(z)))^{3} +  a_1 \gamma g'(\gamma g^{-1}(z))(g'(g^{-1}(z)))^2  - a_1p'(z)(g'(g^{-1}(z)))^3  \nonumber \\
 &{}& \qquad +  a_0 g(\gamma g^{-1}(z))(g'(g^{-1}(z)))^3 - a_0p(z)(g'(g^{-1}(z)))^3 \Bigg)  \nonumber \\
&=& [g(\gamma g^{-1}(g(\gamma g^{-1}(z)))) -  p(g(\gamma g^{-1}(z)))]  \cdot  \left(\frac{\gamma g'(\gamma g^{-1}(z))}{g'(g^{-1}(z))}\right)
+ bh'(z)\nonumber \\
&=& [y(y(z)) -p(y(z))]y'(z) + bh'(z).
\end{eqnarray*}
That  is, $y(z) = g(\gamma g^{-1}(z))$  is  an  analytic  solution  of the equation $(\ref{E:6}).$ The  proof  is  completed.
\end{proof}

\indent We construct an analytic solution of  $(\ref{E:1})$ from an analytic solution of the equation $(\ref{E:6}).$
Assume that $x(z)$ is an analytic solution of the functional differential equation $(\ref{E:1})$ in a neighborhood of the origin.
Since $x(z)$ is analytic in a neighborhood of the origin, $x(z)$ can be represented by a Taylor's series
		$$x(z) =\sum_{n=0}^{\infty} \frac{x^{(n)}(0)}{n!}\cdot z^n = x(0) + x'(0)z + \frac{x''(0)z^2}{2!} + \cdots.$$
\indent Let $p(z) = \sum_{n=0}^{\infty} p_nz^n$ where  $p_n = \frac{p^{(n)}(0)}{n!}$ for $n \geq  2.$ \\
Since $  x'(z)= \frac{1}{b}(y(z) - p(z)), x''(z) = \frac{1}{b}(y'(z)  - p'(z))$ and $a_2x''(z) + a_1x'(z) + a_0x(z)  = x(p(z) + bx'(z)) + h(z),$
 we have $ x'(0)=  -\frac{p_0}{b}, x''(0)  = \frac{1}{b} (\gamma - p_1), x(0) = \frac{a_0p_0 - a_2(\gamma - p_1) + bh_0}{b(a_0 - 1)}$ and $x'''(0) = \frac{a_0p_0 - p_0\gamma -a_1(\gamma - p_1) + bh_1}{ba_2},$ respectively.   \\
By using mathematical induction, we can show that for $m \geq 1$
\begin{align*}
x^{(m+2)}(0)  &=  \frac{1}{a_2} \Bigg( (x(p(z) + bx'(z)))^{(m)}(0) + h^{(m)}(0) - a_1x^{(m+1)}(0) - a_0x^{(m)}(0) \Bigg) \\
&=   \frac{1}{a_2} \Bigg( \Big( \sum_{i=1}^{m} p_{im} (p'(0) + bx''(0), p''(0) + bx'''(0), \ldots, \nonumber  \\
&  \qquad  p^{(m)}(0) + bx^{(m+1)}(0) ) x^{(i)}(p(0) +bx'(0)) \Big) + h^{(m)}(0) - a_1x^{(m+1)}(0) - a_0x^{(m)}(0)  \Bigg),
\end{align*}
where $p_{im}(1 \leq i \leq m)$ is a polynomial with nonnegative coefficients.\\
Therefore, the explicit form of an analytic solution of our equation is
\begin{align}
x(z) &= x(0) + x'(0)z + \frac{x''(0)}{2!} z^2 + \frac{x''(0)}{3!} z^3 + \sum_{n=4}^{\infty} \frac{x^{(n)}(0)}{n!} z^n   \nonumber \\
&=   \frac{a_0p_0 - a_2(\gamma - p_1) + bh_0}{b(a_0 - 1)} + \left( \frac{-p_0}{b}\right) z + \frac{1}{2!} \cdot \left( \frac{\gamma -  p_1}{b} \right) \cdot z^2  \nonumber \\
& \qquad + \frac{1}{3!} \cdot \left( \frac{a_0p_0 - p_0\gamma -a_1(\gamma - p_1) + bh_1}{ba_2}\right) z^3
+ \sum_{m = 1}^{\infty} \frac{\Gamma _m}{(m+3)!} \cdot z^{m+3} \label{eq.solution}
\end{align}
where $\Gamma _m $ denotes $x^{(m+2)} (0).$

In another way, we can construct the solution of the equation $(\ref{E:1})$ by using the fact that a power series $(\ref{E:4})$ is a solution of the equation $(\ref{E:3}).$\\
Let $g^{-1}(z) = \sum_{n = 1}^{\infty} q_n z^n.$ Then
\begin{align}
 z=\sum_{n = 1}^{\infty} \quad \quad  \sum _{\makebox [0 pt ]{$%
{{\scriptscriptstyle l_1+\cdots+l_m = n \atop {%
{m=1, \ldots,n }}}} %
$}}
 q_m c_{l_1} \cdots c_{l_m}  z^n. \label{eq.g^{-1}}
\end{align}
Comparing coefficients in both sides of $(\ref{eq.g^{-1}})$, we have
\begin{align} \label{eq.g2}
q_1 c_1 = 1
\end{align}
and for $n \geq 2,$ we get
\begin{align} \label{eq.g3}
 \sum _{\makebox [0 pt ]{$%
{{\scriptscriptstyle l_1+\cdots+l_m = n \atop {%
{m=1, \ldots,n }}}} %
$}}
 q_m c_{l_1} \cdots c_{l_m} = 0.
\end{align}
From $(\ref{eq.g2})$ and $(\ref{eq.g3})$, we obtain $q_1 = c_1^{-1}, q_2 =-c_1^{-3}c_2, q_3 = c_1^{-4}(2c_1^{-1}c_2^2 -c_3),$ and so on.
Therefore,
\begin{align*}
x(z) &= x(0) +  \frac{1}{b}   \int_{0}^{z} \Bigg( \sum_{n = 1}^{\infty} \qquad  \sum _{\makebox [0 pt ]{$%
{{\scriptscriptstyle l_1+\cdots+l_m = n \atop {%
{m=1, 2, \ldots,n }}}} %
$}} c_{m} \gamma^m q_{l_1} \cdots q_{l_m} s^n - \sum_{n = 0}^{\infty}p_n s^n \Bigg) ds \\
  &= x(0) +  \left( \frac{-p_0}{b}\right) z + \frac{1}{b} \Bigg[\sum_{n = 1}^{\infty} \frac{1}{n+1}  \qquad \Big(  \sum _{\makebox [0 pt ]{$%
{{\scriptscriptstyle l_1+\cdots+l_m = n \atop {%
{m=1, 2, \ldots,n }}}} %
$}}
 c_{m} \gamma^m q_{l_1} \cdots q_{l_m}  - p_n \Big)  \Bigg] z^{n+1}.
\end{align*}
Let $x(z) = \sum_{n=0}^{\infty} x_n z^n$  with $x_0 = x(0), x_1 = \frac{-p_0}{b}$  and
\begin{align} \label{E:x_n}
x_{n+1} = \frac{1}{b} \Bigg[\sum_{n = 1}^{\infty} \frac{1}{n+1}  \qquad \Big(  \sum _{\makebox [0 pt ]{$%
{{\scriptscriptstyle l_1+\cdots+l_m = n \atop {%
{m=1, 2, \ldots,n }}}} %
$}}c_{m} \gamma^m q_{l_1} \cdots q_{l_m}   - p_n \Big)  \Bigg] z^{n+1}.
\end{align}
From $(\ref{E:x_n}),$ we get $x_2 = \frac{1}{2!}\Bigg(\dfrac{ \gamma   - p_1 }{b} \Bigg)$
and $x_3 =\frac{1}{3!}\Bigg(\dfrac{a_0 p_0 - p_0 \gamma  - a_1(\gamma - p_1) + bh_1}{ba_2}\Bigg),$ and so on. These yield the explicit solution as in  $(\ref{eq.solution}).$\\
{\bf Example}
The analytic solution of the equation \\ $(1-2i)x''(z) + (1+i)x'(z) + (3i)x(z) = x(2+i +(2i)z + z^2 + (1+i)x'(z)) + 2 + (2-i)z + z^2$
is
\begin{align*}
x(z) &=   \frac{3i(2+i) - (1-2i)(\gamma -2i) + (1+i)(2)}{(1+i)(3i-1)}  - \left( \frac{ 2+i }{1+i}\right) z + \frac{1}{2!} \cdot \left( \frac{\gamma -  2i}{1+i} \right) \cdot z^2 \\
& \qquad + \frac{1}{3!} \cdot \left( \frac{(3i)(2+i ) - (2+i ) \gamma  - (1+i)(\gamma - 2i) + (1+i) (2-i)}{(1+i) (1-2i) }\right) z^3 + \cdots\\
&=   \frac{10i +3 - \gamma (1-2i)}{2i - 4}  - \left( \frac{2+i}{1+i}\right) z + \frac{1}{2!} \cdot \left( \frac{\gamma -  2i}{1+i} \right) \cdot z^2 \\
& \qquad + \frac{1}{3!} \cdot \left( \frac{9i -2 - \gamma (3+2i )  }{3 - i}\right) z^3
+ \cdots.
\end{align*}
\begin{col}
Some special situations of  the equation $(\ref{E:1})$ yield  the corresponding results as follows.
    \begin{enumerate}
    \item  If $a_2 = a_0 =0, a_1=1$ and $h(z) \equiv 0,$ then the equation $(\ref{E:1})$ reduces to the equation $x'(z) = x(p(z) + bx'(z))$
and we obtain the corresponding result in \cite{pingping}.

    \item   If $a_1 = a_0 = 0, a_2=1, p(z) = az$ and $h(z) \equiv 0,$ then the equation $(\ref{E:1})$ reduces to the equation $x''(z) = x(az + bx'(z))$
and we obtain the corresponding result in \cite{SiWang}.

    \item   If $ a_2 = 0,  p(z) = az, x(z) = z$ and $h(z) \equiv 0,$ then the equation $(\ref{E:1})$ reduces to the equation  $a_0z + a_1x'(z) = x(az + bx'(z))$ and we obtain the corresponding result in  \cite{JSXS}.
  \end{enumerate}      
\end{col}

\section{Acknowledgements}
The first author is supported in part by Development and Promotion of Science and Talents Project (DPST). The second author is supported by National  Research Council of Thailand and Khon Kaen University, Thailand (Grant number: kku fmis(580010)).

\bibliography{mybibfile}

\end{document}